\documentclass[10pt, a4paper]{article}
\usepackage{amssymb}
\usepackage{amsmath}
\usepackage{amsfonts}
\usepackage{amscd}
\usepackage{geometry}
\usepackage{mathdots}
\usepackage[all,cmtip]{xy}
\usepackage{makeidx}
\usepackage{titlesec}
\usepackage{fancyhdr}

\newcommand{\R}{\mathbb{R}}
\newcommand{\Q}{\mathbb{Q}}
\newcommand{\N}{\mathbb{N}}
\newcommand{\Z}{\mathbb{Z}}
\newcommand{\A}{\mathbb{A}}
\newcommand{\F}{\mathbb{F}}

\DeclareMathOperator{\proj}{proj}

\DeclareMathOperator{\Soc}{Soc}
\DeclareMathOperator{\Aut}{Aut}
\DeclareMathOperator{\Ind}{Ind}
\DeclareMathOperator{\Gal}{Gal}
\DeclareMathOperator{\GL}{GL}
\DeclareMathOperator{\PGL}{PGL}
\DeclareMathOperator{\SL}{SL}
\DeclareMathOperator{\PSL}{PSL}
\DeclareMathOperator{\PO}{PO}

\DeclareMathOperator{\PSO}{PSO}
\DeclareMathOperator{\PGO}{PGO}
\DeclareMathOperator{\GO}{GO}
\DeclareMathOperator{\POM}{P\Omega}

\titleformat{\section}[hang]
{\normalfont\filright\large}{\thesection. }{0pt}
{\upshape\bfseries}

\titleformat{\subsection}[hang]
{\itshape}{\thesubsection \ - }{0pt}
{}

\usepackage{amsthm}
\theoremstyle{plain}
\newtheorem{theo}{Theorem}[section]
\newtheorem{prop}[theo]{Proposition}

\newtheorem{coro}[theo]{Corollary}
\theoremstyle{remark}

\theoremstyle{definition}
\newtheorem{defi}[theo]{Definition}

\title{L\"ubeck's classification of representations of finite simple groups of Lie type and the inverse Galois problem for some orthogonal groups}
\author{\small ADRI$\acute{\mbox{A}}$N ZENTENO  \footnote{Departamento de Ciencias B$\acute{\mbox{a}}$sicas, Universidad Aut$\acute{\mbox{o}}$noma Metropolitana. Av. San Pablo No. 180, Edif. H, Col. Reynosa Tamaulipas, CP 02200, Mexico City. \texttt{matematicazg@ciencias.unam.mx}}}

\date{\today}
\begin{document}

\maketitle

\begin{abstract}
In this paper we prove that for each integer of the form $n=4\varpi$ (where $\varpi$ is a prime between 17 and 73) at least one of the following groups: $\POM^\pm_n(\F_{\ell^s})$, $\PSO^\pm_n(\F_{\ell^s})$, $\PO_n^\pm(\F_{\ell^s})$ or $\PGO^\pm_n(\F_{\ell^s})$ is a Galois group of $\Q$ for almost all primes $\ell$ and infinitely many integers $s > 0$.
This is achieved by making use of the classification of small degree representations of finite simple groups of Lie type in defining characteristic due to L\"ubeck and a previous result of the author on the image of the Galois representations attached to RAESDC automorphic representations of $\GL_n(\A_\Q)$.

\end{abstract}

\section{Introduction}

Let $\ell$ be a prime, $r$ be a positive integer and $n$ be an even positive integer different from 8. In \cite{As84}, Aschbacher classified the maximal subgroups of the orthogonal groups $\GO_n^\pm(\F_{\ell^r})$.
Essentially, this classification divides the maximal subgroups into two classes. The first of these consists roughly of subgroups that preserve some kind of geometric structure, so they are commonly called subgroups of geometric type. On the other hand, the second class consists roughly of those absolutely irreducible subgroups that, modulo the central subgroup of scalar matrices, are almost simple and which are not of geometric type. 
We remark that in this paper we will not use the Aschbacher's definition of subgroups of geometric type, but a slight variant which will be more suitable for our purposes (see Section 6 of \cite{Ze18}).

In \cite{KL90}, Kleidman and Liebeck provide a detailed enumeration of the maximal subgroups of geometric type of the finite classical groups of dimension greater than 12, and in \cite{BHR13}, the work of Kleidman and Liebeck is extended to handle dimensions at most 12. 
Using this classification, the local and global Langlands correspondence, and Arthur's work on endoscopic classification of symplectic and orthogonal automorphic representations, in \cite{Ze18}, the author proves that there exist compatible systems $\mathcal{R}= \{ \rho_\ell \}_\ell$ of orthogonal Galois representations $\rho_\ell : \Gal (\overline{\Q} / \Q) \rightarrow \GO_n(\overline{\Q}_\ell)$, $n \equiv 0 \mod 4$, such that for almost all primes $\ell$ the image of $\overline{\rho}_\ell$ (the semi-simplification of the reduction of $\rho_\ell$) cannot be contained in a maximal subgroup of geometric type (in the sense of Definition 6.2 of \cite{Ze18}).
In this paper, by using the main result of \cite{LP98}, we prove that (in fact) the image of $\overline{\rho}_\ell^{\proj}$ (the projectivization of $\overline{\rho}_\ell$) is contained in an almost simple group with socle a finite simple group of Lie type for almost all $\ell$ (i.e., all but finitely many).

In fact, by using the classification of small degree representations of finite simple groups of Lie type in defining characteristic of L\"ubeck \cite{Lub16}, we prove the following result:

\begin{theo}\label{princi}
Let $\mathcal{R} = \{ \rho_\ell \}_\ell$ be a compatible system of $n$-dimensional orthogonal Galois representations $\rho_\ell: G_\Q \rightarrow \GO_n(\overline{\Q}_\ell)$ as in Theorem \ref{rt} below. 
If $n=4\varpi$, where $\varpi$ is a prime number such that $17 \leq \varpi \leq 73$, then for almost all $\ell$ the image of $\overline{\rho}_\ell^{\proj}$ is equal to one of the following groups: 
\begin{equation}\label{realis}
\POM^\pm_n(\F_{\ell^s}), \; \PSO^\pm_n(\F_{\ell^s}), \; \PO_n^\pm(\F_{\ell^s}) , \; \PGO^\pm_n(\F_{\ell^s})
\end{equation} 
for some positive integer $s$.
\end{theo}

As a consequence of this theorem, we have the following result concerning to the inverse Galois problem.

\begin{coro}\label{corpri} 
For each $n$ of the form $4\varpi$, where $\varpi$ is a prime number such that $17 \leq \varpi \leq 73$, we have that at least one of the groups in (\ref{realis}) occurs as a Galois group over $\Q$ for infinitely many primes $\ell$ and infinitely many positive integers $s$. 
\end{coro}

To the best of our knowledge, these orthogonal groups are not previously known to be Galois over $\Q$, except for some cases where $s$ is so small which were studied in \cite{Re}, \cite{MM} and \cite{Zy14}.

\subsection*{Notation}

Through this paper, if $K$ is a perfect field, we denote by $\overline{K}$ an algebraic closure of $K$ and by $G_K$ the absolute Galois group $\Gal(\overline{K} / K)$. Whenever $G$ is a subgroup of a certain linear group $\GL_n(K)$, we write P$G$ for the image of $G$ in the projective linear group $\PGL_n(K)$. In fact, for classical groups in general, we will use the same notation and conventions as in Section 2 of \cite{Ze18}.
Finally, if $G$ is a finite group, we denote by $G^{(i)}$ the $i$-th derived subgroup of $G$ and we use $G^{\infty}$ to denote $\displaystyle \bigcap_{i \geq 0} G^{(i)}$. 


\section{Review of previous results} \label{Se:1}

Let $n$ be an even positive integer and $p,t > n$ be two distinct odd primes such that the order of $t$ mod $p$ is $n$. Denote by $\Q_{t^n}$ the unique unramified extension of $\Q_t$ of degree $n$ and recall that $\Q^{\times}_{t^n} \simeq  t^{\Z} \times \mu_{t^n-1} \times U_1$, where $\mu_{t^n-1}$ is the group of $(t^n-1)$-th roots of unity and $U_1$ is the group of 1-units. 
Let $\ell$ be a prime distinct from $p$ and $t$. We will say that a character 
$\chi_t : \Q^{\times}_{t^n} \rightarrow \overline{\Q}^{\times}_\ell$
is of $O$-\emph{type} if $\chi_t (t) = 1$ and $\chi_t \vert _{\mu_{t^n-1} \times U_1}$ has order $p$.
By local class field theory we can regard $\chi_t$ as a character (which by abuse of notation we will also denote by $\chi_t$) of $G_{\Q_{t^n}}$ or of $W_{\Q_{t^n}}$. Then, we can define the Galois representation 
\[
\rho_t := \Ind^{G_{\Q_t}}_{G_{\Q_{t^n}}}(\chi_t),
\] 
which is irreducible and orthogonal in the sense that it can be conjugated to take values in $\GO_n(\overline{\Q}_\ell)$. Moreover, if $\alpha: G_{\Q_t} \rightarrow \overline{\Q}^\times_\ell$ is an unramified character, the residual representation $\overline{\rho}_t \otimes \overline{\alpha}$ is also irreducible (see Lemma 5.2 of \cite{Ze18}).

\begin{defi}\label{minduced}
Let $p$, $t$, $\ell$ and $\rho_t$ as above. We say that an orthogonal Galois representation 
\[
\rho_\ell: G_\Q \longrightarrow \GO_n(\overline{\Q}_\ell)
\]
is \emph{maximally induced of }$O$-\emph{type} at $t$ of order $p$ if the restriction of $\rho_\ell$ to a decomposition group at $t$ is equivalent to $\rho_t \otimes \alpha$ for some unramified character $\alpha: G_{\Q_t} \rightarrow \overline{\Q}^\times_\ell$.
\end{defi}

From now on, we will assume that $n \geq 10$ (see Remark 6.5 of \cite{Ze18}). The main result of \cite{Ze18} gives us a criterion to know when a compatible system $\mathcal{R} = \{ \rho_\ell \}_\ell$ of orthogonal Galois representations $\rho_\ell: G_\Q \rightarrow \GO_n(\overline{\Q}_\ell)$, which are maximally induced of $O$-type at $t$ of order $p$ for an ``appropriate" couple of primes $(p,t)$, is such that the image of $\overline{\rho}_\ell$ cannot be contained in a maximal subgroup of geometric type (see Definition 6.2 of \emph{loc. cit.}) of an $n$-dimensional orthogonal finite group for almost all $\ell$. 

In order to be more precise, we need to explain what we mean by an appropriate couple of primes.
Let $k,n,N \in \N$ with $n$ as above; and $M$ be an integer greater than $n^4(n+2)!$, $N$, $kn!+1$ and all primes dividing $2\prod_{i=1}^{f}(2^{2i} -1)$ if $n=2^f$ for some $f \in \N$. Let $L_0$ be the compositum of all number fields of degree smaller that or equal to $n!$ which are ramified at most at the primes smaller than or equal to $M$. By Chebotarev's Density Theorem, we can choose two different primes $p$ and $t$ such that: 
\begin{enumerate}
\item $p \equiv 1 \mod n$,
\item $p$ and $t$ are greater than $M$,
\item $t$ is completely split in $L_0$, and
\item $t^{n/2} \equiv -1 \mod p$. 
\end{enumerate}

We remark that the choice of $M$ is slightly different from that of Lemma 6.3 of \cite{Ze18}. Such choice will be used below to avoid the pathologies associated with the finite groups of small order. However, it is easy to see that the following result works, despite this slight modification. 

\begin{theo}\cite[Theorem 6.4]{Ze18}\label{rt}
Let $k$, $n$, $N$, $M$, $p$, $t$ and $L_0$ as above. Let $\mathcal{R} = \{ \rho_\ell \}_\ell$ be a compatible system of orthogonal Galois representations $\rho_\ell: G_\Q \rightarrow \GO_n(\overline{\Q}_\ell)$ such that for every prime $\ell$, $\rho_\ell$ ramifies only at the primes dividing $Nt\ell$. 
Assume that for every $\ell>kn!+1$ a twist of $\overline{\rho}_\ell$ by some power of the cyclotomic character is regular with tame inertia weights at most $k$ (in the sense of Section 3.2 of \cite{ADW14}) and  that for all $\ell \neq p,t$, $\rho_\ell$ is maximally induced of $O$-type at $t$ of order $p$.
Then, for all primes $\ell$ different from $p$ and $t$, the image of $\overline{\rho}_\ell$ cannot be contained in a maximal subgroup of $\GO^{\pm}_n(\F_{\ell^r})$ of geometric type.
\end{theo}

Moreover, for 12-dimensional compatible systems of orthogonal Galois representations as in the previous theorem, it can be proven that the image of $\overline{\rho}^{\proj}_\ell$ is a finite orthogonal group for almost all $\ell$. More precisely, in the second part of Theorem 8.1 of \cite{Ze18}, it is stated that the image of $\overline{\rho}_\ell^{\proj}$ is equal to $\POM^+_{12}(\F_{\ell^s})$, $\PSO^+_{12}(\F_{\ell^s})$, $\PO_{12}^+(\F_{\ell^s})$ or $\PGO^+_{12}(\F_{\ell^s})$ for some  integer $s>0$. However, the arguments in the proof are slightly incorrect because the methods used in \emph{loc. cit.}, cannot distinguish between the orthogonal groups of plus and minus type, due to our poor control on the coefficient fields. So, Theorem 8.1.$ii)$ of \cite{Ze18} has to be weakened as follows:
 
\begin{theo}\label{rib}
Let $\mathcal{R} = \{ \rho_\ell \}_\ell$ be a compatible system of orthogonal Galois representations $\rho_\ell: G_\Q \rightarrow \GO_{12}(\overline{\Q}_\ell)$ as in Theorem \ref{rt}. Then, for almost all $\ell$, we have that the image of $\overline{\rho}_\ell^{\proj}$ is equal to $\POM^\pm_{12}(\F_{\ell^s})$, $\PSO^\pm_{12}(\F_{\ell^s})$, $\PO_{12}^\pm(\F_{\ell^s})$ or $\PGO^\pm_{12}(\F_{\ell^s})$ for some  integer $s>0$.
\end{theo}

The main goal of this paper is to give a refinement of Theorem \ref{rt} in order to prove a similar result to Theorem \ref{rib} for certain high dimensions.


\section{Improvement of Theorem \ref{rt}}\label{Se:2}

In order to give a refinement of Theorem \ref{rt}, we need a more precise description of the maximal subgroups of $\GO^\pm_n(\F_{\ell^r})$ which are not of geometric type. 

As we mentioned in the introduction of this paper, the maximal subgroups of $\GO^\pm_n(\F_{\ell^r})$ were classified essentially by Aschbacher in \cite{As84}.  More precisely, Aschbacher proved that, if $G$ is a maximal subgroup of $\GO^\pm_n(\F_{\ell^r})$ which does not contain $\Omega^\pm_n(\F_{\ell^r})$, then one of the following holds:
\begin{enumerate}
\item $G$ is of geometric type;
\item $G$ is of class $\mathcal{S}$; or
\item $G$ stabilizes a subfield of $\F_{\ell^r}$ of prime index.
\end{enumerate}
It is important to remark that this formulation of the Aschbacher's classification is slightly different from the usual classification given in \cite{As84} or in \cite{KL90} which also consider the maximal subgroups lying in case iii) as of geometric type. However, it is enough for our purposes because, as we will see below, we are not interested in excluding these groups. 
We refer the reader to Section 6 of \cite{Ze18} for Aschbacher's classification formulated in the precise form that we will use it.

\begin{defi}\label{chida}
Let $\ell$ be an odd prime, $r$ be a positive integer and $n$ as above. We will say that a maximal subgroup $G$ of $\GO^\pm_n(\F_{\ell^r})$ is \emph{of class} $\mathcal{S}$ if all the following holds:
\begin{enumerate}
\item $G$ does not contain $\Omega^\pm_n(\F_{\ell^r})$;
\item $G^\infty$ acts absolutely irreducible on $\F_{\ell^r}^n$;
\item $G^\infty$ is not conjugated to a group defined over a proper subfield of $\F_{\ell^r}$;
\item P$G$ is almost simple, i.e., $S \leq \mbox{P}G \leq \Aut(S)$ for some non-abelian simple group $S$. 
\end{enumerate} 
\end{defi}

In \cite{BHR13} the maximal subgroups of class $\mathcal{S}$ are classified for dimensions at most 12. 
This classification is a key ingredient in the proof of Theorem \ref{rib} (see Section 8 of \cite{Ze18}).
Unfortunately, to the best of our knowledge, it is a feature of the subgroups in class $\mathcal{S}$ that they are not susceptible to a uniform description across all dimensions. 

In order to describe (at least partially) the maximal subgroups of $\GO^\pm_n(\F_{\ell^r})$ of class $\mathcal{S}$, we need to recall some standard facts from group theory.
First, recall that the \emph{socle} of a finite group $G$ is defined as the subgroup $\Soc(G)$ of $G$ generated by the minimal normal subgroups of $G$. In particular, if $G$ is an almost simple group, i.e., if $G$ is such that $S \leq G \leq \Aut(S)$ for some non-abelian simple group $S$, then $S = \Soc(G)$ and $S$ is a normal subgroup of $G$. We will say that a group $G$ is \emph{perfect} if it is equals to its commutator subgroup. In particular, if $G$ is an almost simple group, $\Soc(G) = G^{\infty}$ and it is perfect (see Section 3.1 of \cite{Lo17}).

On the other hand, in this paper, a finite group will be called \emph{simple of Lie type} if it is a finite twisted or non-twisted simple adjoint Chevalley group in characteristic $\ell \neq 2,3$. The assumption in the characteristic is in order to avoid the difficulties associated with the Suzuki and Ree groups (see Chapter 3 of \cite{Atl} for the relevant definitions).

It can be proven that the socle of almost all subgroups in class $\mathcal{S}$ is a finite simple group of Lie type. In order to prove this, we need the following useful result of Larsen and Pink \cite{LP98}.

\begin{theo}\label{LaPi}
Let $\ell \neq 2,3$. For every $n$ there exists a constant $J(n)$ depending only on $n$ such that any finite subgroup $G \subseteq \GL_n(\overline{\F}_\ell)$  has normal subgroups $G_3 \subseteq G_2 \subseteq G_1$ with the following properties:
\begin{enumerate}
\item $G_3$ is an $\ell$ group;
\item $G_2/ G_3$ is an abelian group of order not divisible by $\ell$;
\item $G_1/G_2$ is a direct product of finite simple groups of Lie type in characteristic $\ell$;
\item $G/G_1$ is of order at most $J(n)$.
\end{enumerate}
\end{theo}

\begin{prop}\label{lemi}
Let $G$ be a maximal subgroup of $\GO^\pm_n(\F_{\ell^r})$ of class $\mathcal{S}$ and $J(n)$ be the constant in the previous theorem (which depends only on $n$).
Then, if $\vert \mbox{\emph{P}}G \vert > J(n)$, the socle of \emph{P}$G$ is a finite simple group of Lie type in characteristic $\ell$.
\end{prop}

\begin{proof}
We will closely follow the proof of Proposition 6.4 of \cite{Lo17}. Applying Theorem \ref{LaPi} to $G$ we have that there exist normal subgroups $G_3 \subseteq G_2 \subseteq G_1$ of $G$ satisfying the properties $i)$-$iv)$ of the previous theorem.

First, note that as $G_3$ is a solvable normal subgroup of $G$, then P$G_3$ is a normal subgroup of P$G$ (which is almost-simple since $G$ is of class $\mathcal{S}$). It is well known that an almost simple group does not possess non-trivial normal solvable subgroups \cite[Lemma 3.3]{Lo17}, then P$G_3$ is trivial. This implies that $G_3$ is a subgroup of the group of homotheties in $\GL_n(\F_{\ell^r})$ which has order prime to $\ell$, thus $G_3$ is trivial.
A similar argument shows that $G_2 \subseteq \F_{\ell^r}^\times \cdot I_n$.

On the other hand, we remark that P$G_1$ cannot be trivial, for otherwise we would have $\vert \mbox{P} G \vert \leq J(n) \vert \mbox{P}G_1 \vert = J(n)$ which contradicts our hypothesis.  Therefore, P$G_1$ contains $\Soc(G)$ because this is a non-trivial normal subgroup of P$G$.

Finally, as $G_2 \subseteq \F_{\ell^r}^\times \cdot I_n$, P$G_1$ is a quotient of $G_1/G_2$ (hence in particular a direct product of finite simple groups of Lie type in characteristic $\ell$) and, as the outer automorphism group of a simple group is solvable \cite[Theorem 1.3.2]{BHR13}, we have that $\Soc(\mbox{P}G) = (\mbox{P}G_1)^{\infty}$ is a finite simple group of Lie type in characteristic $\ell$.  
\end{proof}

Now, we are ready to prove the following refinement of Theorem \ref{rt}, which is a simple consequence of the previous result.

\begin{prop}\label{ref1}
Let $\mathcal{R} = \{ \rho_\ell \}_\ell$ be a compatible system of Galois representations as in Theorem \ref{rt}. Then, for all $\ell$ different from $2$, $3$, $p$ and $t$, the image of $\overline{\rho}_\ell^{\proj}$ is contained in an almost simple group with socle a finite simple group of Lie type in characteristic $\ell$.
\end{prop}
\begin{proof}
First by Lemma 8.2 of \cite{Ze18}, if $\ell$ is different from 2, 3, $p$ and $t$, we have the following two possibilities for the image of the Galois representations in Theorem \ref{rt}:
\begin{enumerate}
\item the image of $\overline{\rho}_\ell$ is contained in a maximal subgroup of $\GO_n^\pm(\F_{\ell^r})$ of class $\mathcal{S}$; or
\item the image of $\overline{\rho}_\ell^{\proj}$ is contained in a subgroup, which is conjugate to $\POM^\pm_n(\F_{\ell^s})$, $\PSO^\pm_n(\F_{\ell^s})$, $\PO_n^\pm(\F_{\ell^s})$ or $\PGO_n^\pm(\F_{\ell^s})$ for some integer $s>0$ dividing $r$.
\end{enumerate}
From the definition of maximally induced representation, we have that the image of $\overline{\rho}_\ell^{\proj}$ contains an element of order $p$.
Moreover, from \cite[Theorem A]{Col08} we have that $J(n) \leq n^4(n+2)!$. Then, if we are in the first case, as we have chosen $p > M$ (in particular greater than $n^4(n+2)!$), we can apply Proposition \ref{lemi} to conclude that the socle of the almost simple group containing the image of
$\overline{\rho}_\ell^{\proj}$ is a finite simple group of Lie type in characteristic $\ell$. The second case is trivial. 
 \end{proof}
 

\section{Representation theory of finite simple groups of Lie type}\label{Se:3}

In this section we will recall some results about algebraic groups and their defining characteristic representations that we will need in this paper. 

From now on, we will assume $\ell \geq 5$ and let $q=\ell^e$ be a power of $\ell$. We define the $q$-\emph{Frobenius map} of $\GL_n(\overline{\F}_\ell)$ as the automorphism $F_q : \GL_n(\overline{\F}_\ell) \rightarrow \GL_n(\overline{\F}_\ell)$ given by $(a_{i,j}) \mapsto (a_{i,j}^q)$. Let $G$ be a connected reductive simple algebraic group over $\overline{\F}_\ell$ of simply connected type and rank $m$. We will say that a homomorphism  $F:G(\overline{\F}_\ell) \rightarrow G(\overline{\F}_\ell)$ is a $\emph{standard Frobenius map}$ if there exist an injective homomorphism $\iota: G(\overline{\F}_\ell) \rightarrow \GL_n(\overline{\F}_\ell)$ (for some $n$) and a power of $\ell$ ($q=\ell^e$) such that $\iota (F(\gamma)) = F_q (\iota (\gamma))$ for all $\gamma \in G(\overline{\F}_\ell)$, and we will say that $F: G(\overline{\F}_\ell) \rightarrow G(\overline{\F}_\ell)$ is a \emph{Frobenius endomorphism} if some power of it is a standard Frobenius map. 
Recall that to each $G$ we can associate a connected Dynkin diagram which determines the (Lie) type of $G$. We remark that a Frobenius map is completely characterized by the choice of an automorphism of the Dynkin diagram of $G$ together with a real number $q$ which, in our setting, is an integral power of $\ell$. 

It is well known that to a finite simple group of Lie type P$G$ in characteristic $\ell$ (i.e., a finite twisted or non-twisted simple adjoint Chevalley group in characteristic $\ell \neq 2,3$) we can attach a connected reductive simple algebraic group $G$ over $\overline{\F}_\ell$ of simply connected type and a Frobenius endomorphism $F:G(\overline{\F}_\ell) \rightarrow G(\overline{\F}_\ell)$ such that P$G \cong G^F/Z$, where $G^F  := \{ \gamma \in G(\overline{\F}_\ell) : F(\gamma) = \gamma \}$ is the group of fixed points of $F$ and $Z$ is the center of $G^F$. 

Our interest in the groups $G$ and $G^F$ associated to P$G$ comes from the fact that projective representations of P$G$ in characteristic $\ell$ are the same as linear representations of $G^F$ in characteristic $\ell$ (see \cite{Ste16}, p. 49, items (ix) and (x)), which in turn can be constructed by restricting algebraic representations of $G$ to $G^F$ as follows. 

Let $T$ be a maximal torus of $G$, $X \cong \Z^m$ its character group and $Y \cong \Z^m$ its co-character group. Let $\{ \alpha_1, \ldots, \alpha_m \} \subset X$ be a set of simple roots for $G$ with respect to $T$ and $\alpha_i^\vee \in Y$ the coroot corresponding to $\alpha_i$ for each $i=1, \ldots, m$. The numbering of the simple roots we use is that of \cite{Bou02}.
Viewing $X \otimes \R$ as Euclidean space we can define the \emph{fundamental weights} $\omega_1 \ldots, \omega_m \in X \otimes \R$ as the dual basis of $\alpha_1^\vee, \ldots, \alpha_m^\vee$. 
It is possible to define a partial ordering in $X$ by declaring that a weight $\omega$ is smaller than a weight $\omega'$ (in symbols $\omega \prec \omega'$) if and only if $\omega'-\omega$ is a non-negative linear combination of simple roots. We will say that a weight $\omega \in X$ is \emph{dominant} if it is a non-negative linear combination of the fundamental weights.

Let $M$ be a finite-dimensional $G(\overline{\F}_\ell)$-module. Considering this as $T$-module, we have a decomposition $M= \bigoplus _{\omega \in X} M_\omega$ into weight spaces $M_\omega = \{ v \in M: vt = \omega(t)v \mbox{ for all }t \in T \}$. The set of weights $\omega \in X$ with $M_\omega \neq \{ 0 \}$ is called the \emph{set of weights} of $M$. 
From a result of Chevalley  (see Section 31.3 of \cite{Hum75}) we have that, if $M$ is irreducible, then the set of weights of $M$ contains a unique element $\lambda$ such that for all weight $\omega$ of $M$ we have $\omega \prec \lambda$. This $\lambda$ is called the \emph{highest weight} of $M$, it is dominant and $\dim(M_\lambda) =1$. Moreover, we have that each irreducible $G(\overline{\F}_\ell)$-module $M$ is determined (up to isomorphism) by its highest weight and for each dominant weight $\lambda \in X$  there is an irreducible $G(\overline{\F}_\ell)$-module $M(\lambda)$ with highest weight $\lambda$. 

We will say that a dominant weight $\lambda = a_1 \omega_1 + \cdots + a_m \omega_m \in X$ is $q$-\emph{restricted} if $0 \leq a_i \leq q-1$ for  $1 \leq i \leq m$. Then, the relationship between the $\overline{\F}_\ell$-representations of $G^F$ and the algebraic representations of $G$ is described by the following result of Steinberg \cite{Ste63}.

\begin{theo}\label{Ste}
Any irreducible $\overline{\F}_\ell$-module for $G^F$ is isomorphic to the restriction of the $G(\overline{\F}_\ell)$-module $M(\lambda)$ to $G^F$ for some $q$-restricted weight $\lambda$. 
Moreover, if $\lambda$ and $\mu$ are $q$-restricted and not equal, then the restrictions of $M(\lambda)$ and $M(\mu)$ to $G^F$ are not isomorphic as $\overline{\F}_\ell [G^F]$-modules.
\end{theo}

We remark that the Frobenius map $F:G(\overline{\F}_\ell) \rightarrow G(\overline{\F}_\ell)$ determines a permutation of the simple roots and this permutation gives rise to a symmetry $\sigma$ of the Dynkin diagram of $G$. Then, by the previous result, we have that $G^F$ has $q^m$ absolutely irreducible representations over $\overline{\F}_\ell$ (up to isomorphism) and each of them can be written over $\F_{q^\epsilon}$, where $\epsilon$ is the order of $\sigma$ (the symmetry associated to $F$).

Now we will describe more finely the structure of the simple modules $M(\lambda)$. More precisely, we will explain how all highest weight modules $M(\lambda)$ of $G$ can be constructed out of those with $\ell$-restricted highest weights.

Let $x \mapsto x^\ell$ be a field automorphism of $\overline{\F}_\ell$. This can be used to construct a canonical endomorphism of the algebraic group $G$ that we will denote by $F_0$. Twisting the $G$-action on a $G(\overline{\F}_\ell)$-module $M$ with $F_0^i$, $i \in \Z_{\geq 0}$, we get another $G(\overline{\F}_\ell)$-module that we will denote by $M^{(i)}$. Now we can enunciate the Twisted Tensor Product Theorem of Steinberg \cite{Ste63}.

\begin{theo}\label{tutu}
Let $\lambda_0, \cdots, \lambda_r$ be $\ell$-restricted weights associated with $G$. Then, as $G(\overline{\F}_\ell)$-modules,
\[
M(\lambda_0 + \ell \lambda_1 + \cdots + \ell^r \lambda_r) \cong M(\lambda_0) \otimes M(\lambda_1)^{(1)} \otimes \cdots \otimes M(\lambda_r)^{(r)}.
\]
\end{theo}

This result, together with Theorem \ref{Ste}, allows us to reduce the study of the $\overline{\F}_\ell$-representations of $G^F$ to the study of the $G(\overline{\F}_\ell)$-modules that can be constructed out of those with $\ell$-restricted highest weights.
More precisely, let $\mathcal{M}$ be the set of $\ell$-\emph{restricted modules} for $G^F$ which consists of the restrictions to $G^F$ of the $M(\lambda)$ for all $\ell$-restricted weights $\lambda$, and $\mathcal{M}_i := \{ M^{(i)}: M \in \mathcal{M} \}$. It is easy to see that $\mathcal{M}_0 = \mathcal{M} = \mathcal{M}_e$ and that $\mathcal{M} \cap \mathcal{M}_i = \{ 1 \}$ for $i=1, \ldots,e-1$. Then, any irreducible $\overline{\F}_\ell$-module for $G^F$ has the form:
\[
M_0 \otimes M_1 \otimes \cdots \otimes M_{e-1},
\]
where $M_i \in \mathcal{M}_i$ for all $i$. Moreover, the $\vert \mathcal{M}\vert^e = q ^m$ possibilities are pairwise non-isomorphic. 

In \cite{Lub16} all irreducible representations of $G$ (of type different from $A_1$), of dimension below some bound, are determined. Then, the $\ell$-restricted modules for the groups $G^F$ and $G$ of dimension up to at least 300 are those listed in Table 2 and Tables 6.6-6.53 of \cite{Lub16} for which the entries in the column $\lambda$ of those tables are all less than $\ell$. On the other hand, when $G$ is of type $A_1$, we have the following description. Let $V_1$ be the natural module of $\SL_2(\F_q)$ and $V_{n+1}$ be the $(n+1)$-symmetric power of $V_1$. When $n \geq 1$, $V_{n+1}$ is a faithful module for $\SL_2(\F_q)$, if $n$ is odd, and a faithful module for $\PSL_2(\F_q)$, if $n$ is even.
It can be proven that, if $\ell>n$, then $V_{n+1}$ is absolutely irreducible and self-dual for $\SL_2(\F_q)$. The modules $V_{n+1}$ for $\ell>n$ are just the $\ell$-restricted modules for $\SL_2(\F_q)$. In the language of highest weights, such modules are of the form $M(a_1 \omega_1)$, with $0\leq a_1 \leq \ell - 1$ and their dimension is $a_1+1$. See Section 5.3 of \cite{BHR13} for more details. 

Finally, we can conclude that the results discussed in this section, together with the tables in \cite{Lub16},  are sufficient to obtain all projective representations of finite simple groups of Lie type (i.e., Chevalley and twisted Chevalley groups) of dimension up to at least 300. 


\section{Proof of the main results} \label{Se:4}

In this section we prove Theorem \ref{princi}, which extends Theorem \ref{rib} to representations of dimension $n = 4\varpi$, where $\varpi$ is a number prime between 17 and 73. 

First, we remark that in view of Proposition \ref{ref1} we are interested in the question of whether the groups $\PGO^\pm_n(\F_{\ell^r})$ contain some almost simple group with socle a finite simple group of Lie type in characteristic $\ell$. 
So, according to part ii) of Definition \ref{chida}, we are just interested in the absolutely irreducible representations $\overline{\rho}: G^F \rightarrow \Omega_n^\pm(\F_{q^\epsilon})$ for P$G := \Soc (Im (\overline{\rho}_\ell^{\proj}))$ a finite simple group of Lie type in characteristic $\ell$, i.e., representations of the form $M(\lambda)$ for certain $q$-restricted weight $\lambda$ (see Theorem \ref{Ste} and the paragraph below this). 

As we remark in the previous section, we have that $\overline{\rho}$ is equivalent to a representation with module $M_0 \otimes \cdots \otimes M_{e-1}$, where $M_i \in \mathcal{M}_i$ for all $i$. In particular, if $n \leq 4\cdot 73<300$, the image of $\overline{\rho}$ is $\GO^\pm_n(\F_{q^\epsilon})$-conjugated to the image of a tensor product of the representations listed in \cite{Lub16} or a tensor product of representations coming from $G$ of type $A_1$. 
In fact, as the image of $\overline{\rho}$ may be conjugate to a subgroup preserving a symmetric form, we are just interested in the self-dual modules $M(\lambda) \in \mathcal{M}$ because $M_0 \otimes \cdots \otimes M_{e-1}$ is self-dual if and only if each $M_i$ is self-dual (see Section 5.1.2 of \cite{BHR13}).

 We remark that the highest weight of the irreducible module $M(\lambda)^*$ (the dual of $M(\lambda)$) is $-\omega_0 \lambda$, where $\omega_0$ is the longest word in the Weyl group of $G$. This remark implies that $M(\lambda)$ is self-dual if and only if $\lambda = -\omega_0 \lambda$ that is, $\lambda$ is invariant under the graph automorphism of the Dynkin diagram associated to $G$. In particular; since $-\omega_0 = \mbox{Id}$ for the groups of type $A_1$, $B_n$, $C_n$, $D_n$($n$ even), $E_7$, $E_8$, $F_4$ and $G_2$; all irreducible modules for these groups are self-dual.

On the other hand, recall that the Frobenius-Schur indicator allows us to decide when a self-dual representation is orthogonal or symplectic. If this indicator is 1, the representation is orthogonal and if it is $-1$, the representation is symplectic (see Lemma 78 and Lemma 79 of \cite{Ste16} and Section 6.3 of \cite{Lub16}).  

Finally, we remark that in our case, if $M_i$ and $M_j$ are two self-dual modules both symplectic or both orthogonal, the tensor product $M_i \otimes M_j$ is orthogonal; and if one of $M_i$ and $M_j$ is symplectic and the other is orthogonal, then $M_i \otimes M_j$ is symplectic (see Section 1.9 of \cite{BHR13}).

\begin{proof}[Proof of Theorem \ref{princi}]
According to the previous discussion, it is enough to prove that the socle of the image of $\overline{\rho}^{\proj}_\ell$ is a finite simple group P$G$ of type $D_{2\varpi}$ or ${^2D_{2\varpi}}$. In other words, we need to prove that the only irreducible $4\varpi$-dimensional representations $\overline{\rho}: G^F \rightarrow \Omega_{4\varpi}^\pm(\F_{q^\epsilon})$, occur when $G$ is of type $D_{2\varpi}$ (see also $\S$ 8.2 of \cite{Ze18}).

First, we assume that $G$ is of a Lie type different from $A_1$. As in Section 6.4 of \cite{Lub16}, in order to determine all irreducible representations in defining characteristic for groups of a fixed Lie type, we first need to compute all factorizations of $n$ into factors greater than 1. Such factorizations are $4\varpi = 4\cdot \varpi = 2 \cdot 2 \cdot \varpi = 2 \cdot 2 \varpi$.
According to Tables 6.6-6.53 of \cite{Lub16} there are no irreducible representations of dimension 2 and the only irreducible representations of dimension 4 occur when $G$ is of type $A_3$ or type $B_2$. However, the representations of groups of type $A_3$ are of the form $M(\omega_1)$ and $M(\omega_3)$, so these are no self-dual, and the only representations of $B_2$ of dimension $\varpi$ occur when $\ell = 7$ (resp. $\ell=11$) and $\varpi = 71$ (resp.  $\varpi = 61$).  Thus, if we assume that $\ell > 4\varpi$, there are no tensor decomposable representations of dimension $4\varpi$. Moreover, according to Table 2 and Tables 6.6-6.53 of \cite{Lub16}, if $\ell > 19$ the only irreducible tensor indecomposable self-dual representations of dimension $4\varpi$ are the natural representations of $C_{2\varpi}$ and $D_{2\varpi}$, but the natural representation of $C_{2\varpi}$ is symplectic, hence when the Lie type of $G$ is different from $A_1$ and $\ell>4\varpi$ the only irreducible $4\varpi$-dimensional orthogonal representation occurs when $G$ is of type $D_{2\varpi}$.

Finally, following the discussion of Section 5.3 of \cite{BHR13}, when $G$ is of type $A_1$, the tensor decomposable representations should be of the form $V_{n/m} \otimes V_{n/m}^\sigma \otimes \cdots \otimes V_{n/m}^{\sigma^\alpha}$ were $\sigma$ is the field automorphism $x \rightarrow x^q$. As the twist of a representation by a field automorphism does not change the dimension, then $\alpha$ is such that $(n/m)^\alpha = 4\varpi$.
So, there are no tensor decomposable representations of dimension $4 \varpi$ in this case. Moreover, as the $\ell$-restricted modules $V_{n-1}$ are of the form $M((4 \varpi - 1) \omega_1)$, when $\ell > 4 \varpi$, we have that its Frobenius-Schur indicator is $-1$. Then, the tensor indecomposable representations are all symplectic. Hence, there are no self-dual orthogonal representations of dimension $4 \varpi$ when $G$ is of type $A_1$ and $\ell > 4\varpi$, which concludes our proof. 
\end{proof}

\begin{proof}[Proof of Corollary \ref{corpri}] 
Given a couple of primes $(p,t)$, chosen as in Section \ref{Se:1}, we have from Corollary 7.4  of \cite{Ze18} that there exists a compatible system $\mathcal{R} = \{ \rho_\ell \}_\ell$ of orthogonal Galois representations $\rho_\ell: G_\Q \rightarrow \GO_n(\overline{\Q}_\ell)$ as in Theorem \ref{rt}.  Then, by Theorem \ref{princi} we have that, if $\ell$ is greater than $4\varpi$ and different from $p$ and $t$, the image of $\overline{\rho}_\ell^{\proj}$ is one of the following groups: $\POM^\pm_n(\F_{\ell^s})$, $\PSO^\pm_n(\F_{\ell^s})$, $\PO_n^\pm(\F_{\ell^s})$ or $\PGO^\pm_n(\F_{\ell^s})$ for some $s>0$.

Finally, by Chebotarev's Density Theorem, there are infinite ways to choose the couple of primes $(p,t)$. Then, Corollary 7.4 of \emph{loc. cit.}, implies that there are infinite compatible systems $\mathcal{R}_k = \{ \rho_{\ell,k} \}_{\ell,k}$ ($k \in \N$)  such that, for a fixed prime $\ell>n$, the size of the image of $\overline{\rho}_{\ell,k}^{\proj}$ is unbounded for running $k$, because we can choose $p$ as large as we please so that elements of larger and larger orders appear in the inertia images.
\end{proof}

\subsection*{Acknowledgments}
I want to give special thanks to the anonymous referee whose  insightful comments and suggestions were fundamental to improve the presentation and readability of this paper.


\end{document}